\documentclass[12pt,twoside,a4paper]{amsart}
\usepackage{amssymb}
\date{\today}

\begin{document}

\title[]
{An explicit   Koppelman formula for $dd^c$ and
Green currents on $\P^n$}

\def\la{\langle}
\def\ra{\rangle}
\def\End{{\rm End}}

\def\deg{\text{deg}\,}


\def\al{\text{\tiny $\aleph$}}
\def\w{\wedge}
\def\db{\bar\partial}
\def\dbar{\bar\partial}
\def\ddbar{\partial\dbar}
\def\loc{{\rm loc}}
\def\rank{\text{rank}}
\def\R{{\mathbb R}}
\def\C{{\mathbb C}}
\def\w{{\wedge}}
\def\P{{\mathbb P}}
\def\cp{{\mathcal C_p}}
\def\bl{{\mathcal B}}
\def\A{{\mathcal A}}
\def\B{{\mathcal B}}
\def\Cn{\C^n}
\def\ddc{dd^c}
\def\bbox{\bar\square}
\def\G{{\mathcal G}}
\def\od{\overline{D}}
\def\ot{\leftarrow}
\def\loc{\text{loc}}
\def\D{{\mathcal D}}
\def\M{{\mathcal M}}
\def\dbarb{\bar\partial_b}
\def\mr{\paren{-\rho}}
\def\Hom{{\rm Hom\, }}
\def\codim{{\rm codim\,}}
\def\Tot{{\rm Tot\, }}
\def\Im{{\rm Im\, }}
\def\Lop{{\mathcal L}}
\def\Nop{{\mathcal N}}
\def\K{{\mathcal K}}
\def\Ker{{\rm Ker\,  }}
\def\Dom{{\rm Dom\,  }}
\def\Z{{\mathcal Z}}
\def\E{{\mathcal E}}
\def\can{{\rm can}}
\def\Ok{{\mathcal O}}
\def\Pop{{\mathcal P}}
\def\L{{\mathcal L}}
\def\Q{{\mathcal Q}}
\def\Re{{\rm Re\,  }}
\def\Res{{\rm Res\,  }}
\def\Aut{{\rm Aut}}
\def\L{{\mathcal L}}
\def\U{{\mathcal U}}

\def\sr{\stackrel}

\newtheorem{thm}{Theorem}[section]
\newtheorem{lma}[thm]{Lemma}
\newtheorem{cor}[thm]{Corollary}
\newtheorem{prop}[thm]{Proposition}

\theoremstyle{definition}

\newtheorem{df}{Definition}

\theoremstyle{remark}

\newtheorem{preremark}{Remark}
\newtheorem{preex}{Example}

\newenvironment{remark}{\begin{preremark}}{\qed\end{preremark}}
\newenvironment{ex}{\begin{preex}}{\qed\end{preex}}

\numberwithin{equation}{section}

\date{\today}

\author{Mats Andersson}

\address{Department of Mathematics\\Chalmers University of Technology and the University of G\"oteborg\\
S-412 96 G\"OTEBORG\\SWEDEN}

\email{matsa@math.chalmers.se}



\thanks{The author was
  partially supported by the Swedish Natural Science
  Research Council. The author is also grateful to the Institut Mittag-Leffler
(Djursholm, Sweden) where parts of this work was carried out.}

\begin{abstract} 
We compute a quite explicit   Koppelman formula for $dd^c$ on
projective space, and obtain
Green currents for positive closed currents.
\end{abstract}

\maketitle

\section{Introduction}

Let $X$ be a smooth projective variety and 
let $\theta$ be a positive closed $(p,p)$-current on $X$.
A  $(p-1,p-1)$-current $g$ is called a Green current for $\theta$ if
\begin{equation}\label{green}
dd^c g+\theta =A,
\end{equation}
where $A$ is a smooth form.

Green currents for  Lelong currents
$[Z]$ of analytic cycles $Z$  are  of fundamental importance in Arakelov geometry,
see, e.g.,  \cite{GS}, \cite{BGS}, \cite{SABK}, \cite{BY1},
and the survey article \cite{Hein}. 
Green currents for more general
$\theta$ are used, e.g.,  in complex dynamics, see \cite{DiSi} and \cite{DiSi2}.

\smallskip

This paper is an elaboration of the second half\footnote{The first part 
is published as \cite{Apl}.} of  \cite{Aarxiv}.
We construct  a (positive)   integrable kernel $K(\zeta,z)$ of
bidegree $(n-1,n-1)$ that is smooth outside  the diagonal and of log type
along $\Delta$ (see Section~\ref{plsec} for the definition) 
and a smooth kernel $P(\zeta,z)$  of bidegree $(n,n)$,  such that
\begin{equation}\label{basform}
dd^c K+ [\Delta]=P
\end{equation}
in the current  sense;  i.e., $K$ is  a Green current for $[\Delta]$
on $\P^n_\zeta\times\P^n_z$. 
From \eqref{basform} we get, for smooth forms
$\theta$, the   Koppelman type formula
\begin{multline}\label{kman}
dd^c\int_\zeta K\w \theta +\theta=\\
\int P(\zeta,z)\w\theta(\zeta)
+d\int_\zeta K\w d^c\theta
-d^c\int_\zeta  K \w d\theta-
\int_\zeta K\w dd^c\theta
\end{multline}
for the $dd^c$-operator.
If $\theta$ is a  $(k,k)$-form such that $d\phi=0$, then  $d^c\theta=0$ as well, 
and if
\begin{equation}\label{skata}
g=\int_\zeta K(\zeta,z)\w\theta(\zeta),\quad  A=\int_\zeta P(\zeta,z)\w\theta(\zeta), 
\end{equation}
then $dd^c g+\theta=A$.

In the meantime we have learnt that such Green currents  $K$
were  obtained in  essentially the same way already in \cite{BGS}, Section~6, 
so our main contribution is 
that we compute our  kernels $K$ and $P$ quite explicitly.
If $\theta$ is the Lelong current $[Z]$ for some cycle $Z$, then the function
$g(z)$ in \eqref{skata} is well-defined and smooth in  $X\setminus Z$, and it 
is proved in \cite{BGS} (Lemma~1.2.2) that it is in fact a Green current
for $Z$  of log type along $Z$.
In \cite{DiSi} it is proved that the same formula provides a Green current
for any positive closed $(k,k)$-current $\theta$.
With our particular choice of $K$, the Koppelman formula extends to any 
current $\theta$, see below.

\smallskip
Our starting point is the following  result.

\begin{prop}\label{wmminusett}
Let $f$ be a holomorphic section of a Hermitian vector bundle $E\to X$ of rank $m$
over some manifold $X$.
If the zero cycle $Z$ of $f$ has codimension $m$, i.e., $f$ intersects the zero section
transversally,  then  there is a current $W_{m-1}$,
smooth outside $Z$ and of log type, such that
\begin{equation}\label{corett}
dd^c W_{m-1}+[Z]=c_m(D_E),
\end{equation}
where $c_m(D_E)$ is the top Chern form of $E$.
\end{prop}

Thus  $W_{m-1}$ is a Green current for $[Z]$.
The proof (in  Section~6 in \cite{BGS} and in \cite{Apl})
is based on  ideas in  \cite{BC}. 
See Section~\ref{plsec} below for a discussion, and for an expression
for the current  $W_{m-1}$.

\smallskip

Let $z=[z_0,\ldots,z_n]$ be the usual homogeneous coordinates on $\P^n$ and let
$$
\omega=dd^c\log|z|
$$
be the standard K\"ahler form on $\P^n$ so that
\begin{equation}\label{pelly}
\int_{\P^n}\omega^n=1.
\end{equation}
We will see that 
$$
\eta=\sum_{j=0}^n z_j\frac{\partial}{\partial\zeta_j}
$$
is a holomorphic section of the bundle
$H=T^{1,0}(\P^n_\zeta)\otimes\Ok(-1)_\zeta\otimes\Ok(1)_z$ over 
$\P^n_\zeta\times\P^n_z$ that  vanishes to first order precisely on the
diagonal. From the standard metric on $\P^n$ we get
a natural Hermitian metric   $\|\ \|$ on $H$, and in particular, see Section~\ref{pn},
\begin{equation}\label{etanorm}
\|\eta\|^2=\frac{|\zeta|^2|z|^2-|z\cdot\bar\zeta|^2}{|z|^2|\zeta|^2}=
\frac{|\zeta\w z|^2}{|z|^2|\zeta|^2},
\end{equation}
which is like the square of the  distance  between the points
$[z]$ and $[\zeta]$ on $\P^n$.
From  Proposition~\ref{wmminusett}
we thus get  $K=W_{n-1}$ and  $P=c_n(D_H)$ so that 
\eqref{basform} holds.

We say  that a differential form $\xi(\zeta,z)$ on $\P^n_\zeta\times\P^n_z$ is invariant
if $\xi(\phi(\zeta),\phi(z))=\xi(\zeta,z)$ for each isometric automorphism
$\phi$ of $\P^n$, i.e.,  mapping induced by a orthogonal mapping on $\C^{n+1}$. 
For instance, $\|\eta\|$ is invariant.
Here is our main theorem.

\begin{thm}\label{stendumt} 
The kernels $P$ and $K$  so defined satisfy \eqref{basform}
and moreover: 

\smallskip
\noindent (i)\ The kernel $P$ is 
\begin{equation}\label{pek}
P=\sum_0^n \omega_z^k\w\omega_\zeta^{n-k},
\end{equation}
and it induces the orthogonal projection onto the harmonic forms.

\smallskip
\noindent (ii)\  $K$ is positive, of log type,  and it can be written
\begin{equation}\label{korm}
K=  \sum_{j=0}^{n-1}  \Big[\log(1/\|\eta\|)\frac{\gamma_j^1}{\|\eta\|^{2j+2}}  +
 \frac{\gamma_j^2}{\|\eta\|^{2j+2}}\Big],
\end{equation}
where
$\gamma_j^i$  are smooth (real-analytic) invariant
forms on $\P^n_\zeta\times \P^n_z$ which are $\Ok(\|\eta\|^2)$.  
In particular, $K(\zeta,z)$ is invariant.

\smallskip
\noindent (iii) The Koppelman formula \eqref{kman} holds  for all currents $\theta$.
\end{thm}

It is well-known that
the harmonic forms (with respect to $\omega$) are precisely the forms
$\alpha \omega^k$ for $\alpha\in\C$,   and   a $(k,k)$-current 
$\xi$ is orthogonal to the harmonic forms if and only if
$$
\int_{\P^n} \xi\w\omega^{n-k}=0.
$$
In particular,  any form that is $d$-exact or $d^c$-exact must
be orthogonal to the harmonic forms. 
In view of \eqref{pelly} and   \eqref{pek},  $P$
induces the orthogonal projection.

Thus if $Z$ is a cycle of pure codimension  $p$ in $\P^n$ with degree
$$
\deg Z=\int_{\P^n} [Z]\w \omega^{n-p},
$$
then
\begin{equation}\label{alvar}
g=\int_\zeta K(\zeta,z)\w [Z]
\end{equation}
solves the Green equation
$$
dd^c g+[Z]=\deg(Z)\omega^p.
$$

\smallskip

In \cite{BY1} and \cite{BY2}  a Green current for $[Z]$ in $\P^n$ is
obtained by means of 
the Levine form for the subspace $z=w$ in $\P^{2n+1}$;
the Green current appears as the value at the origin of a current-valued  analytic
function, involving expressions for homogeneous forms $f_j$
that define $Z$. Such a representation is  the purpose in  those  paper(s). 
It seems that  the resulting current
can be represented by an integral operator with integrable kernel;
however we do not know  whether one can compute this  kernel explicitly.



\begin{remark}
In  \cite{BGS},  p.~913, it is proved that for  any
smooth projective variety $X$  there exists a Green current
$K$  for the diagonal $\Delta$ in $X\times X$
such that $K$ is smooth outside $\Delta$  and of log type.
In particular,  $K$ is integrable on $X\times X$. Thus one gets a 
Koppelman formula on $X$, though not explicit, and a Green current 
for any cycle $Z$ in $X$.
\end{remark}

In this paper, 
$$
d^c=\al(\dbar-\partial), \quad \al=i/2\pi, 
$$
so that
$$
dd^c=2\al\partial\dbar=\frac{i}{\pi}\partial\dbar.
$$

\section{A generalized Poincar\'e-Lelong formula}\label{plsec}

We first discuss a slightly more general  form of Proposition~\ref{wmminusett}.
Let $E\to X$ be a Hermitian vector bundle 
of rank $m$ over a complex (compact) manifold $X$,
let $D_E$ be the Chern connection on $E$,  and  let  $c(D_E)$ be the
associated Chern form, i.e., $c(D_E)=\det(\al\Theta_E+I)$, 
where $\Theta_E=D^2_E$ is the curvature tensor.
We let $c_k(D_E)$ denote the component of $c(D_E)$ of bidegree $(k,k)$.

Let $f$ be   a holomorphic section of $E$, and
assume that $Z=\{f=0\}$ has codimension $p$. 
Then we have  an analytic $p$-cycle 
$$
Z^p=\sum \alpha_j Z_j^p,
$$
where $Z_j^p$ are the irreducible components of $\{f=0\}$ and
$\alpha_j$ are the Hilbert-Samuel multiplicities of $f$.
Moreover, $f$  defines a trival line bundle
$S$ over $X\setminus Z$, and we let $Q=E/S$ be the quotient bundle, 
equipped with the induced 
Hermitian metric, and let $D_Q$ and $c(D_Q)$
be the corresponding Chern connection and Chern form, respectively.

\begin{prop}\label{matspl} 
The form $c(D_Q)$ is locally integrable in $X$ and 
the corresponding current $C(D_Q)$ is closed  in  $X$.
There is  an explicit current $W_{p-1}$ of bidegree $(p-1,p-1)$ and 
of order zero in $X$,   smooth in $X\setminus Z$ and
of logarithmic type along $Z$,   such that
\begin{equation}\label{bas}
dd^cW_{p-1} + [Z^p] = c_p(D_E)-C_p(D_Q).
\end{equation}
If in addition, $E^*$ is Nakano negative, 
one can choose $W_{p-1}$ to be positive.
\end{prop}

Since $Q$ has rank $m-1$, $C_m(D_Q)=0$ and thus we get
\eqref{corett}.
Proposition~\ref{matspl}  is a special case of
a  more general formula in \cite{Apl} (Theorems~1.1 and 1.2).
For the statement about logarithmic type, see Lemma~\ref{lt} below.

\begin{remark}
In \cite{Apl} it is stated that one can choose $W_{p-1}$ positive  where
$|f|<1$, given that $E^*$ is Nakano negative. 
Since $X$ is assumed to be compact 
here, we can modify the metric in $E$ so that $|f|\le 1$ everywhere.
\end{remark}

We shall now  describe  $W_{p-1}$ following  the presentation in \cite{Apl},
and we refer to \cite{Apl} for more details and arguments.
We   introduce the exterior algebra bundle
$\Lambda=\Lambda(T^*(X)\oplus E\oplus E^*)$. Any 
section $\xi\in\E_{k,q}(X,E)$,
i.e., smooth $(k,q)$-form with values in $E$,
corresponds to a section $\tilde\xi$ of $\Lambda$:
If $\xi=\xi_1\otimes e_1+\ldots +\xi_m\otimes e_m$ in a local frame
$e_j$ for $E$, then we let 
$\tilde\xi=\xi_1 \w e_1+\ldots +\xi_m \w e_m$.
In the same way,   $a\in\E_{k,q}(X,\End E)$ 
is  identified with
$$
\widetilde a=\sum_{jk} a_{jk} \wedge e_j\wedge e^*_k,
$$
where  $e_j^*$ is the dual frame, if  $a=\sum_{jk} a_{jk}\otimes e_j\otimes e^*_k$
with respect to these frames.
The connection $D_E$  extends in a unique  way to a linear  mapping
$D\colon \E(X,\Lambda)\to\E(X,\Lambda)$ which is
an anti-derivation with respect to the
wedge product in $\Lambda$, and acts   as 
the exterior differential $d$ on the $T^*(X)$-factor.
If  $\xi$ is a form-valued section of  $E$, then
$
\widetilde{D_E\xi}=D\tilde\xi,
$
and if  $a\in\E_k(X,\End E)$, then 
\begin{equation}\label{lapp}
\widetilde{D_{\End E}a}=D\widetilde a,
\end{equation}
see, e.g., \cite{Apl}.
Since $D_{\End E}I_E=0$, here $I$ denotes the identity endomorphism on $E$, and
by Bianchi's  identity, $D_{\End E}\Theta_E=0$, we have from 
\eqref{lapp} that 
\begin{equation}\label{enol}
D\widetilde\Theta_E=0 \quad \text{and}\quad  D\tilde I=0.
\end{equation}
We let  $\tilde I_m=\tilde I^m/m!$  and  use the same notation 
for other forms in the sequel.
Any form $\omega$ with values in $\Lambda$ can be  written 
$\omega=\omega'\w\tilde I_m+\omega''$ uniquely, 
where $\omega''$ has lower  degree in $e_j,e^*_k$. If we make the definition
\begin{equation}\label{fermionic}
\int_e\omega =\omega',
\end{equation}
then this integral is linear and 
\begin{equation}\label{skodon}
d\int_e\omega=\int_e D\omega.
\end{equation}
We have that 
\begin{equation}\label{laban}
c(D_E)=\int_e(\al\widetilde{\Theta}_E+\tilde I_E)_m=
\int_e e^{\al\widetilde{\Theta}_E+\tilde I_E}.
\end{equation}
Recall that $D_E=D_E'+\dbar$, where  $D'_E$  is the    $(1,0)$-part of $D_E$. 
It follows that we also have the decomposition $D=D'+\dbar$.

Let $\sigma$ be the section of $E^*$ over $X\setminus Z$ with minimal norm such that
$f\cdot\sigma=1$. 
In   $X\setminus Z$ we have the  formula 
(Proposition~4.2 in \cite{Apl})
\begin{equation}\label{cq}
c(D_Q)=
\int_e f\w\sigma \w (\tilde I+\al\tilde\Theta
-\al Df\w\dbar\sigma)_{m-1}.
\end{equation}

In a suitable resolution of singularities 
$\nu\colon\tilde X\to X$ we may assume that 
$\nu^*f=f^0f'$, where $f^0$ is (locally)
a holomorphic function, in fact a monomial in
suitable coordinates,  and 
$f'$ is a non-vanishing section (of $\nu^*E)$.  
Then $\nu^*\sigma=(1/f^0)\sigma'$,  where
$\sigma'$ is smooth, and it follows that
$$
\nu^*\big(\sigma\w f\w (Df\w\dbar\sigma)^{k-1}\big)=
\sigma'\w f'\w  (Df'\w\dbar\sigma')^{k-1}
$$
is smooth. Since  $c(D_Q)$ is closed in $X\setminus Z$,
thus $\nu^*c(D_Q)$ has a smooth and closed extension across
the singularity. In particular it is locally integrable, therefore
its push-forward is locally integrable, so we have:

\begin{lma} The form  $c(D_Q)$ is locally integrable  and its natural
current extension $C(D_Q)$ is closed.
\end{lma}

By the usual Poincare-Lelong formula we have
$$
c(D_S)=1+dd^c\log(1/|f|)
$$
outside $Z$, i.e., $1+dd^c\log(1/|f'|)$ in the resolution. It follows that
also $c(D_S)\w c(D_Q)$ is locally integrable. If
capitals denote the natural current extensions, then
\begin{equation}\label{CCC}
-dd^cV=c(D_E)-C(D_S)C(D_Q),
\end{equation}
where $V$ is the locally integrable form
\begin{equation}\label{vformel}
V=\sum_{\ell=1}^{m-1}\frac{(-1)^{\ell}}{2\ell}
\int_e f\wedge\sigma\wedge (\tilde I+\al\tilde\Theta-\al Df\w\dbar\sigma)_{m-1-\ell}
\wedge (-\al Df\w\dbar\sigma)_\ell.
\end{equation}
Finally, if 
$$
W= \log(1/|f|)    C(D_Q)-V,
$$
then its component  $W_{p-1}$ satisfies \eqref{bas}. 

\smallskip
We will be particularily  interested in the case when $p=m$.
For degree reasons then no $\tilde I$ can come into play so 
\begin{multline}\label{wdef}
W_{m-1}=
\log(1/|f|)\int_ef\w\sigma\w(\al\tilde\Theta+\al Df\w\dbar\sigma)_{m-1}\\
-\sum_{\ell=1}^{m-1}\frac{(-1)^{\ell}}{2\ell}
\int_e f\wedge\sigma\wedge (\al\tilde\Theta-\al Df\w\dbar\sigma)_{m-1-\ell}
\wedge (-\al Df\w\dbar\sigma)_\ell.
\end{multline}
If $E$ is positive in the sense that $E^*$ is Nakano negative, then
$C_{m-1}(D_Q)$ is a positive current  and 
\begin{equation}\label{positiv}
W'_{m-1}=W_{m-1}+\alpha C_{m-1}(D_Q)
\end{equation}
is positive, for a large enough  positive constant $\alpha$,
see Section~7 in \cite{Apl}.


\begin{remark}
If $E$ is a line bundle, i.e., $m=1$, then $V=0$, and since  $\sigma\cdot f=1$
we have that $W=\log(1/|f|)$, 
$C_1(Q)=0$,  and hence \eqref{bas}   is just the classical
Poincar\'e-Lelong formula.
\end{remark}

Recall that a $(k,k)$-current  $w$ is of {\it logarithmic type} along
a subvariety $Z$ if the following holds: There exists a surjective mapping 
$\pi\colon \tilde X\to X$ such that $\tilde Z=\pi^{-1}Z$ has normal crossings
in $\tilde X$, 
$\pi$ has surjective differential  in $\tilde X\setminus \tilde Z$,
$w=\pi_*\tilde w$, where $\tilde w$ is smooth in 
$\tilde X\setminus \tilde Z$, and locally $\tilde w$  is of the form
\begin{equation}\label{logtyp}
\sum_j \gamma''_j  \log|s_j| + \gamma',
\end{equation}
where $s$ is a suitable local coordinate system,  
$\gamma'$ is a smooth form,  and
$\gamma''_j$ are  closed smooth forms.

\begin{lma}\label{lt}
The current $W$  is of logarithmic type
along $Z$.
\end{lma}

\begin{proof} Let $\nu\colon \tilde X\to X$ be the resolution above.
Then $\nu^* C(D_Q)$ is smooth and closed,  
$\nu^* V$ is smooth, and $\nu^*\log|f|=\log|f^0|+\log|f'|$,
where $f^0$ locally is a polynomial whose zero set is $\nu^{-1}Z$. 
In view of 
\eqref{wdef} thus $W$ is of logarithmic type along $Z$.
\end{proof}

\section{Explicit Koppelman formulas  on  $\P^n$}\label{pn}

The line bundle  $\Ok(k)$ over $\P^n$, whose
sections are naturally identified by $k$-homogeneous functions $\xi$
on $\C^{n+1}\setminus\{0\}$, 
will be denoted, for typographical reasons, by $L^k$.
We have the natural Hermitian norm
$$
\|\xi([z])\|^2=\frac{|\xi(z)|^2}{|z|^{2k}}.
$$

Recall that a differential form $\alpha$ on $\C^{n+1}$
is projective, i.e., the pullback under $\C^{n+1}\setminus\{0\}\to\P^n$ of a form
on $\P^n$, 
if and only if $\delta_z\alpha=\delta_{\bar z}\alpha
=0$, where
$\delta_z$ is interior multiplication with $\sum_0^n z_j(\partial/\partial z_j)$ 
and $\delta_{\bar z}$ is its conjugate.


\begin{lma}\label{linje}
If $D_{L^r}'$ is the $(1,0)$-part of the Chern connection on 
$L^r$ and $g$ is a section,
expressed as an $r$-homogeneous function  in $z$,  then
\begin{equation}\label{blad}
D_{L^r}' g= |z|^{2r}\partial\frac{g}{|z|^{2r}}=\partial g-r g\partial\log|z|^2,
\end{equation}
and 
\begin{equation}\label{blad2}
\al\Theta_{L^r}=-r\al\dbar\partial\log|z|^2=r\omega.
\end{equation}
\end{lma}

Notice that since $g$ is $r$-homogeneous, i.e., $g(\lambda z)=\lambda^r g(z)$, we have that
$$
r\lambda^{r-1} g(z)=\sum_0^n z_j\frac{\partial g}{\partial z_j}(\lambda z),
$$ 
and thus $\delta_z(\partial g)=r g(z)$.
Since furthermore $\delta_z\partial\log|z|^2=1$, 
the right hand side of \eqref{blad} is indeed a projective form.

\begin{proof}
We may assume that $g$ is (locally) holomorphic.
Then
$$
\partial \|g\|^2=\langle D'_{L^r}g,g\rangle=\frac{(D'_{L^r}g)\bar g}{|z|^{2r}},
$$
but also
$$
\partial\|g\|^2=\partial\frac{|g|^2}{|z|^{2r}}=
|z|^{2r}\partial\frac{g}{|z|^{2r}}\cdot \bar g/|z|^{2r}.
$$
Combining,  we get \eqref{blad}.
Now,
$
\Theta_{L^r}g=\dbar D'_{L^r} g= -r(\dbar\partial\log|z|^2) g,
$
which gives \eqref{blad2}.
\end{proof}

We are now going to compute the currents obtained from 
Proposition~\eqref{wmminusett}  with $H$ and $\eta$ as in the the introduction.
In order not to mix up  with the usual tangent bundle,
we introduce an abstract copy of  $H$, cf.,  \cite{BB}.
Let $L^r_\zeta$ be the pullback of $L^r\to\P^n_\zeta$ to $\P^n_\zeta\times\P^n_z$.
Furthermore, let $\C^{n+1}\to \P^n_\zeta\times\P^n_z$  be the trivial bundle taken with 
the natural metric, and
consider   the quotient bundle $\C^{n+1}/\zeta\C$, 
where $\zeta\C$ is $L_\zeta^{-1}$, i.e., the pullback of the tautological \
line bundle over $\P^n_\zeta$.
We define 
$$
H=(\C^{n+1}/\zeta\C)\otimes L_z
$$
over  $\P^n_\zeta\times\P^n_z$ equipped  with the induced metric.
Since $\C^{n+1}$ has trivial metric it has vanishing curvature, and therefore the
quotient $\C^{n+1}/\zeta\C$ is positive   in the sense that its dual is 
Nakano negative, cf., Proposition~7.1 in \cite{Apl}. It follows that 
$H$ are positive  in the same sense, since  $L$ is positive.

A section $w$ of $H$ is represented by a mapping $\C^{n+1}_\zeta\times\C^{n+1}_z\to\C^{n+1}$
that is $0$-homogeneous in $\zeta$ and $1$-homogeneous in $z$.
In particular we have the
global holomorphic section $\eta(\zeta,z)=z$, which vanishes to first order on the diagonal
$\Delta$. The dual of $\C^{n+1}/\zeta\C$ is the subbundle of $(\C^{n+1})^*$
that is orthogonal to $\bar\zeta\C$. Therefore,
a  section of the dual bundle $H^*$ can be represented by a mapping
$w\colon \C^{n+1}_z\times\C^{n+1}_\zeta\to(\C^{n+1})^*$ that is 
$0$-homogeneous in $\zeta$, $-1$-homogeneous
in $z$, and such that $w\cdot\zeta=0$. 

\smallskip
Let $e_0,\ldots,e_n$ be  the trivial global frame (basis)  for $\C^{n+1}$ above, and let
$e_j^*$ be its dual basis for $(\C^{n+1})^*$. 
If $\xi=\xi\cdot e=\sum\xi_je_j$ is a section of $\C^{n+1}/\zeta\C$, then 
its norm is equal to the norm of the orthogonal projection onto the orthogonal complement of $L^{-1}_\zeta=\zeta\C$,
$$
\frac{|\zeta|^2\xi\cdot e-(\xi\cdot\bar\zeta)\zeta\cdot e}{|\zeta|^2}.
$$
Thus if $\xi\sim\xi\cdot e $ is a section of $H$, i.e., the functions $\xi_j$ are in 
addition $1$-homogeneous in $z$, we have
$$
\|\xi\|^2=\frac{|\zeta|^2|\xi|^2-|\xi\cdot\bar\zeta|^2}{|z|^2|\zeta|^2};
$$
here and in the sequel we use $\|\xi\|$ to distinguish the norm of the section from
$|\xi|$, denoting the norm of the corresponding vector-valued function on $\C^{n+1}$.
In particular, \eqref{etanorm} holds.
Moreover, 
$$
s=\frac{|\zeta|^2\bar z\cdot e^*-(\zeta\cdot\bar z)\bar\zeta\cdot e^*}{|z|^2|\zeta|^2}.
$$
is the section of $H^*$ with minimal (since $\|s\|=\|\eta\|$)  norm such that 
$s\cdot\eta=\|\eta\|^2$.

\smallskip
Notice that a form $\alpha$ on $\C^{n+1}_\zeta\times\C^{n+1}_z$
is projective if and only if $\delta_z\alpha=\delta_{\bar z}\alpha
=\delta_\zeta\alpha=\delta_{\bar\zeta}\alpha=0$.
We can write a form-valued section $\xi$ of $H$ as $\xi=\xi\cdot e=\sum\xi_j\otimes e_j$
where $\xi_j$ are projective forms.

We need expressions for $D_H$ and $\Theta_H$. Let $\omega_z$ and
$\omega_\zeta$ be the
K\"ahler forms on $\P^n_z$ and $\P^n_\zeta$, respectively.

\begin{prop}\label{gruff} 
If $\xi\cdot e$ is a section of $H$, then
\begin{equation}\label{blad4}
D'_H(\xi\cdot e)=\partial\xi\cdot e-\frac{\xi\cdot\bar\zeta}{|\zeta|^2}d\zeta\cdot e
-\frac{\partial|z|^2}{|z|^2}\otimes \xi\cdot e.
\end{equation}
Moreover, 
\begin{equation}\label{blad5}
\al\Theta_H=
\al  d\zeta\cdot e\w \dbar\frac{\bar\zeta\cdot e^*}{|\zeta|^2}
+\omega_z\otimes e\cdot e^*.
\end{equation}
\end{prop}

Here $\partial w\cdot e^*=\sum_0^n \partial w_j\otimes e_j^*$ etc. 
Notice that indeed $d\zeta\cdot e$ is a projective 
$H$-valued form, since it can be written
$$
\Big(d\zeta -\frac{\bar\zeta\cdot d\zeta}{|\zeta|^2}\zeta\Big)
\cdot e
$$
and each $d\zeta_j-(\bar\zeta\cdot d\zeta/|\zeta|^2)\zeta_j$
is projective.

\begin{proof}
First assume that $\xi$ is just a section of  the bundle $F=\C^{n+1}/\zeta\C$ 
over $\P^n_\zeta$. Since $d$ is the Chern connection on $\C^{n+1}$, the Chern connection
on $\xi$ is equal to $d$ acting on the representative of $\xi$ in $F$ that is
orthogonal to $\zeta\C$. Since                                              
$$
\pi(\xi\cdot e)=\frac{\xi\cdot\bar\zeta}{|\zeta|^2}\zeta\cdot e,
$$
is the orthogonal projection $\C^{n+1}\to\zeta\C$, 
we get the formula
$$
D'_F(\xi\cdot e)=\partial\xi\cdot e-\frac{\xi\cdot\bar\zeta}{|\zeta|^2}d\zeta\cdot e.
$$
Since $H=F\otimes L_z$ now
\eqref{blad4} follows from Lemma~\ref{linje}.
Using that $\Theta_E\xi=\dbar D'\xi$ for holomorphic
$\xi$, now \eqref{blad5} follows as well.
\end{proof}

\begin{lma}\label{ee}
 Let $\alpha$ be a form with values in $\Lambda(H\otimes H^*)$, and let
$e'$ denote a local frame for $H$. If $e$ is the standard basis for $\C^{n+1}$ as above, then 
\begin{equation}\label{skol}
 \int_{e'} \alpha =\int_e \frac{\zeta\cdot e\w\bar\zeta\cdot e^*}{|\zeta|^2}\w\alpha,
\end{equation}
where $\alpha$ on the right hand side is any  form with values in 
$\Lambda(\C^{n+1}\oplus(\C^{n+1})^*)$ that
represents $\alpha$. 
\end{lma}

\begin{proof}
Let $[\zeta]$ be an arbitrary point on $\P^n_\zeta$. After applying an  
isometric automorphism of $\P^n$,  we may assume that 
$\zeta= (1,0,0,\ldots)$. If we choose the basis  $e_j'=e_j$, $j=1,2,\ldots$, then 
\eqref{skol} is immediate.
\end{proof}

\begin{proof}[Proof of Theorem~\ref{stendumt}]
We already know that \eqref{basform} holds, so let us now compute $P$.
If $d\zeta\cdot e$ from now on  denotes $\sum_0^n d\zeta_j\w e_j$ 
etc,   we have that 
$$
\al\tilde\Theta_H=
-\al d\zeta\cdot e\w \dbar\frac{\bar\zeta}{|\zeta|^2}\cdot e^*
+\omega_z\w  \tilde I,
$$
where $\tilde I=\sum_0^n e_j\w e_j^*$; the change of sign,  compared to \eqref{blad5},
is because  $e_j\w d\zeta_k=-d\zeta_k\w e_j$. In view of Lemma~\ref{ee} we have
$$
P=c_n(D_H)=\int_{e'}(\al\tilde\Theta_H)_n=
\int_e\frac{\zeta\cdot e\w\bar\zeta\cdot e^*}{|\zeta|^2}\w 
\Big(\frac{-\al d\zeta\cdot e\w d\bar\zeta\cdot e^*}{|\zeta|^2}
+\omega_z\w \tilde I \Big)_n.
$$
Since this formula as well as  \eqref{pek}  are  invariant  we can assume that
$\zeta=(1,0,\ldots,0)$. Then \eqref{pek}  follows by a simple computation. However,
for further reference we prefer a more direct computational argument.
It is easy to check that 
\begin{equation}\label{dd}
\frac{1}{\al(k+1)}\delta_{\bar\zeta}\delta_\zeta(-\al d\zeta\cdot e\w d\bar\zeta\cdot e^*)_{k+1}=
\zeta\cdot e\w\bar\zeta\cdot e^*\w(-\al d\zeta\cdot e\w d\bar\zeta\cdot e^*)_k
\end{equation}
and that
\begin{equation}\label{ddd}
\frac{1}{\al(k+1)|\zeta|^{2k+2}}\delta_{\bar\zeta}\delta_\zeta\beta^{k+1}=\omega_\zeta^k,
\end{equation}
if   $\beta=\al\partial\dbar|\zeta|^2$.
Hence
\begin{multline*}
P=\sum_{k=0}^n\int_e
\frac{\zeta\cdot e\w\bar\zeta\cdot e^*}{|\zeta|^2}\w 
\big(\frac{-\al d\zeta\cdot e\w d\bar\zeta\cdot e^*}{|\zeta|^2}\big)_k\w\tilde I_{n-k}\w\omega_z^{n-k}=\\
\sum_{k=0}^n\frac{1}{\al(k+1)|\zeta|^{2k+2}}
\delta_{\bar\zeta}\delta_\zeta\int_e
(-\al d\zeta\cdot e\w d\bar\zeta\cdot e^*)_{k+1}\w\tilde I_{n-k}\w\omega_z^{n-k}=\\
\sum_{k=0}^n\frac{1}{\al(k+1)|\zeta|^{2k+2}}
\delta_{\bar\zeta}\delta_\zeta \beta^{k+1}\w\omega_z^{n-k}=
\sum_{k=0}^n\omega_\zeta^k\w\omega_z^{n-k}.
\end{multline*}

\smallskip
We now turn our attention to the kernel $K$. 
According to \eqref{wdef} and Lemma~\ref{ee},
\begin{multline*}
K=W_{n-1}=\\
\log(1/\|\eta\|)\int_{e}\frac{\zeta\cdot e\w\bar\zeta\cdot e^*}{|\zeta|^2}
\w\eta\w\sigma\w(\al\tilde\Theta_H-\al D\eta\w\dbar\sigma)_{n-1}-\\
\sum_{\ell=1}^{n-1}
\frac{(-1)^{\ell}}{2\ell}\int_{e}\frac{\zeta\cdot e\w\bar\zeta\cdot e^*}{|\zeta|^2}\w
\eta\w\sigma\w(\al\tilde\Theta_H-\al D\eta\w\dbar\sigma)_{n-1-\ell}
\w(-\al Df\w\dbar\sigma)_\ell.
\end{multline*}
Now
$$
\sigma=\frac{s}{\|\eta\|^2}=\frac{|\zeta|^2\bar z\cdot e^*-(\zeta\cdot\bar z)\bar\zeta\cdot e^*}
{|\zeta\w z|^2},
$$
and since $\eta=z\cdot e$ we therefore have that
$$
\frac{\zeta\cdot e\w\bar\zeta\cdot e^*}{|\zeta|^2}
\w\eta\w\sigma=
\frac{\zeta\cdot e\w\bar\zeta\cdot e^*\w z\cdot e\w\bar z\cdot e^*}{|\zeta\w z|^2}=
\frac{\Ok(|\zeta\w z|^2)}{|\zeta\w z|^2}=
\frac{\Ok(\|\eta\|^2)}{\|\eta\|^2}.
$$
If $\sim$ denotes equality after multiplication with this factor, we have
$$
\al\tilde\Theta_H\sim
\frac{\al d\zeta\cdot e\w d\bar\zeta\cdot e^*}{|\zeta|^2} +\omega_z\w\tilde I,
$$
$$
D\eta=D'(z\cdot e)\sim dz\cdot e-\frac{z\cdot\bar\zeta}{|\zeta|^2}d\zeta\cdot e,
$$
and
$$
\dbar\sigma\sim\frac{|\zeta|^2d\bar z\cdot e^*-(\zeta\cdot\bar z)d \bar \zeta\cdot e^*}
{|\zeta\w z|^2}.
$$
If
$$
\tau=\frac{\big(|\zeta|^2 dz\cdot e-(z\cdot\bar\zeta)d\zeta\cdot e\big)\w
\big(|\zeta|^2 d\bar z\cdot e^*-(\bar z\cdot\zeta)d\bar \zeta\cdot e^*\big)}
{|\zeta|^4|z|^2},
$$
thus
$$
D\eta\w\dbar\sigma\sim \frac{\tau}{\|\eta\|^2}.
$$
From the expression for $K$ above, with the modification  \eqref{positiv}  to obtain
a positive kernel, 
we get  the representation \eqref{korm}, where
\begin{equation}\label{pebby}
\gamma_j^i=c_j^i\int_e \frac{\zeta\cdot e
\w\bar\zeta\cdot e^*\w z\cdot e\w\bar z\cdot e^*}{|\zeta|^2|z|^2}
\w \tau^j\w (\al\tilde \Theta_H)^{n-1-j}
\end{equation}
for some constants  $c_j^i$.
Thus $\gamma_j^i$ are smooth and $\Ok(\|\eta\|^2)$.
To see the invariance let  $\phi$ be a unitary mapping (matrix) on $\C^{n+1}$. First notice that
if $\zeta$ and $z$ are replaced by $\phi\zeta$ and $\phi z$, then 
$|\zeta|^2$, $\omega_z$,  $\zeta\cdot\bar z$ etc are unchanged. Moreover,
 $\zeta\cdot e$, $d\zeta \cdot e$, $\bar\zeta\cdot e^*$ etc,  will
become same expressions, but with $e$ and $e^*$ replaced by $\phi^t e$ and
$\bar\phi^t e^*$, respectively. However, since $\phi$ is unitary,
$\bar\phi^t e^*$ is the dual basis of $\phi^t e$. Since 
\eqref{pebby} is independent of the choice
of frame $e$, it follows that $\gamma_j^i$ is invariant. Since
$\eta$ is invariant, it follows that $K$ is invariant.

\smallskip

It remains to prove part (iii).
By duality we have to show that the (dual) operators map  smooth forms to smooth
forms. Let $p$ be a fixed  point in $\P^n$. 
If $n=1$, then for each point $z$, except for the antipodal point, there is a unique
isometric isomorphism $\phi_z$ such that $z\mapsto p$ and $p\mapsto z$.
If $n>1$, then  for each point $z$ outside the  hyperplane of antipodal points
there is such a mapping,  and it is unique if we require that it is the identity
on the orthogonal complement of the complex line through $p$ and $z$.
It is clear that $\phi_z$ so defined depends smoothly on $z$.
Therefore, for $z$ outside the exceptional hyperplane,
$$
\int_\zeta
K(\zeta,z)\w \psi(\zeta)=\int_\zeta K(\phi_z^{-1}(\zeta),\phi_z^{-1}(z))\w\psi(\zeta)=
\int_\xi K(\xi,p)\w\phi_z^*\psi(\xi),
$$
and it follows that  the integral depends smoothly on $z$
if $\phi$ is smooth.
\end{proof}

\section{A further computation}

Let us calculate $K$ a little further. Since we are mainly interested in the action on
$(*,*)$-forms, we only care about the components that have bidegree $(*,*)$ in $z$.
To simplify even more  let us restrict ourselves to the component $K_0$ which has bidegree
$(0,0)$ in $z$. 
Letting
$$
m=-\frac{\al d\zeta\cdot e\w d\bar\zeta\cdot e^*}{|\zeta|^2}
$$
we then have 
$
\al\tilde\Theta_H\sim m
$
and
$$
-\al D\eta\w\dbar\sigma\sim \frac{|\zeta\cdot\bar z|^2}{|\zeta\w z|^2} m.
$$
Noting that
$$
1+\frac{|\zeta\cdot \bar z|^2}{|\zeta\w z|^2}=1+\frac{|\zeta\cdot \bar z|^2}{|\zeta|^2|z|^2-|\zeta\cdot\bar z|^2}=
\frac{|\zeta|^2|z|^2}{|\zeta\w z|^2},
$$
we also have
$$
\al\tilde\Theta_H-\al D\eta\w\dbar\sigma\sim \frac{|\zeta|^2|z|^2}{|\zeta\w z|^2} m.
$$

\begin{lma}
If $\A$ denotes the component of $c_{n-1}(D_Q)$ which is $(0,0)$ in $z$, then
\begin{multline*}
\A=
\int_e 
\frac{\zeta\cdot e\w\bar\zeta\cdot e^*\w z\cdot e\w\bar z\cdot e^*}
{|\zeta\w z|^2} \w
\Big(\frac{|\zeta|^2|z|^2}{|\zeta\w z|^2} m\Big)_{n-1} 
=\\
c_n\Big(\frac{|z|^2}{|\zeta\w z|^2} \Big)^{n-1}\al ^{n-1}(n-1)! a\wedge\bar a,
\end{multline*} 
where
$$
a=\sum_{j<k}(\zeta_j z_k-\zeta_k z_j)\frac{\partial}{\partial\zeta_k}\neg \frac{\partial}{\partial\zeta_j}\neg
d\zeta_0\w\ldots\w d\zeta_n
$$ 
($\neg$ denotes interior multiplication) and $c_n=\pm 1$.
\end{lma}

\begin{proof}
In fact, 
letting $\delta_z$ and $\delta_{\bar z}$ temporarily denote interior multiplication with 
$z\cdot(\partial/\partial\zeta)$ and $\bar z\cdot(\partial/\partial\bar \zeta)$, respectively,
by a computation as in the proof of Proposition~\ref{stendumt} above we have that
\begin{multline*}
\int_e 
\zeta\cdot e\w\bar\zeta\cdot e^*\w z\cdot e\w\bar z\cdot e^*\w (-\al  d\zeta\cdot e\w d\bar\zeta\cdot e^*)_{n-1}
=\\
\frac{1}{\al^2 n(n+1)}\delta_{\bar z}\delta_z\delta_{\bar\zeta}\delta_\zeta
\int_e (-\al  d\zeta\cdot e\w d\bar\zeta\cdot e^*)_{n+1}=\\
\frac{1}{\al^2 n(n+1)}\delta_{\bar z}\delta_z\delta_{\bar\zeta}\delta_\zeta\beta^{n+1}=
(n-1)!\al^{n-1} a\w\bar a.
\end{multline*}
\end{proof}

Since $c_{n-1}(D_Q)$ is closed it  follows that $\A$ is closed.
This can also easily be verified directly.
Summing up we get the formula
\begin{multline*}
-\frac{1}{2}K_0=\\
\Big[\log\Big(\frac{|\zeta\w z|^2}{|\zeta|^2|z|^2}\Big)
+\sum_{\ell=1}^{n-1}\frac{(-1)^{\ell}}{\ell}
\frac{(n-1)!}{(n-1-\ell)!\ell!}\Big(\frac{|\zeta\cdot \bar z|^2}{|\zeta|^2|z|^2}\Big)^\ell\Big]\w\A.
\end{multline*}
\smallskip

Finally we write this kernel in the  affine coordinates
$\zeta'=(\zeta_1',\ldots,\zeta_n')$ and $z'=(z_1',\ldots,z_n')$.
We  have the transformation rules
$$
|\zeta|^2\mapsto 1+|\zeta'|^2,\quad |z|^2\mapsto 1+|z'|^2,
\quad
\zeta\cdot\bar z\mapsto 1+\zeta'\cdot\bar z',\quad
|\zeta\w z|^2\mapsto  |\zeta'-z'|^2+|\zeta'\w z'|^2.
$$
Furthermore,
$$
a=\sum_1^n (\zeta_j'-z_j')\widehat{d\zeta_j'}=\sum_1^n(\zeta'_j-z_j')\frac{\partial}{\partial\zeta_j'}\neg d\zeta'_1\w
\ldots\w d\zeta'_n.
$$
In affine coordinates we therefore have (suppressing the primes for simplicity)
\begin{multline*}
-\frac{1}{2}K_0=\\
\Big[\log\Big(\frac{|\zeta-z|^2+|\zeta\w z|^2}{(1+|\zeta|^2)(1+|z|^2)}\Big)+\\
\sum_{\ell=1}^{n-1}\frac{(-1)^{\ell}}{\ell}
\frac{(n-1)!}{(n-1-\ell)!\ell!}  \Big(\frac{|1+\zeta\cdot \bar z|^{2}}{(1+|\zeta|^2)(1+|z|^2)}\Big)^\ell \Big]\w
\\
 \Big(\frac{(1+|z|^2)}{|\zeta-z|^2+|\zeta\w z|^{2n-2}}\Big)^{n-1}c_n\al^{n-1}(n-1)! a\wedge\bar a.
\end{multline*}

\section{Green currents for $Z$ in terms of defining functions}\label{kopp}

Let $f_1,\ldots, f_m$ be homogeneous polynomials in $\C^{n+1}$, let $Z^p$  be
the union of the irreducible components of their common zero set $Z\subset \P^n$ of
lowest  codimension $p$,  and let
$$
\|f(\zeta)\|^2 = \sum_1^m \frac{|f_j(\zeta)|^2}{|\zeta|^{2d_j}}.
$$
If all $d_j=d$, then
$$
\|f(\zeta)\|^2 =\frac{|f(\zeta)|^2}{|\zeta|^{2d}}=\frac{\sum_j |f_j(\zeta)|^2}{|\zeta|^{2d}}.
$$
We want to find a Green current for $[Z^p]$ expressed in the functions $f_j$.
To this end let $E_j$ be distinct trivial line bundles with basis elements $e_j$ and
consider $f=\sum f_j e_j$ as a section of the bundle
$$
E=L^{d_1}\otimes E_1 \oplus\cdots\oplus L^{d_m}\otimes E_m. 
$$

\smallskip
We first assume that $p=m$.  Since
$$
c(D_E)=c(L^{d_1})\w\cdots \w c(L^{d_m})=\bigwedge (1+d_j\omega)
$$
we have that   
\begin{equation}\label{bez}
c_m(D_E)=d_1\dots d_m \omega^m =(\deg Z) \omega^m.
\end{equation}
From Proposition~\ref{wmminusett} we get
the Green current $g=W_{m-1}$ for $Z$, solving
\begin{equation}\label{fullsnitt}
dd^c g+[Z]=d_1\cdots d_m \omega^m.
\end{equation}

\begin{prop}\label{hoppla}
This  Green current  for $[Z]$ has the form
$$
g=W_{m-1}=\log(1/\|f\|)\sum_{k=1}^{m}\frac{\gamma_k^1}{\|f\|^{2k}}+
\sum_{k=1}^{m}\frac{\gamma_k^2}{\|f\|^{2k}},
$$
where $\gamma_k^i$ are smooth forms that are  $\Ok(\|f\|^2)$.
If all $d_i=d$, then
\begin{multline}\label{polly}
g=\log(|z|^{d}/|f|)\sum_{\ell=0}^{m-1} d^{m-1-\ell}\omega^{m-1-\ell}\w
(dd^c\log|f|)^\ell +\\
\sum_{\ell=0}^{m-1}c_\ell
 d^{m-1-\ell}\omega^{m-1-\ell}\w
(dd^c\log|f|)^\ell.
\end{multline}
\end{prop}

One can check that the second sum in \eqref{polly} is  closed,
so that already the first sun  is a Green current for $[Z]$. 
This current is  the well-known 
Levin form $L(f)$, cf., p.\ 30 in \cite{BY1}.
It is of course quite easy to verify directly that
$$
dd^c L(f)+[Z]=d^m\omega^m.
$$

\begin{proof}[Proof of Proposition~\ref{hoppla}]
The formula \eqref{wdef}  gives  an explicit expression for
$g$ as soon as we  have explicit
expressions for the  associated sections  $\sigma$, $D_Ef$ and $\Theta_E$.
To begin with 
\begin{equation}\label{citrus}
\sigma= \frac{1}{\|f\|^2} \sum_1^m \frac{\overline{f_j(\zeta)}}{|\zeta|^{2d_j}} e_j^*,
\end{equation}
where  $e_i^*$ are the dual basis elements of $E_i^*$, and
$$
\al\widetilde\Theta_E =\sum_i\al\Theta_{L^{d_i}}\w
e_i\w e_i^*=\sum_1^m d_i\omega\w e_j\w e_i^*=\omega\w\sum_i d_i e_i\w e_i^*.
$$
If all $d_j=d$, then 
\begin{equation}\label{citrusd}
\sigma= \frac{\sum_1^m \overline{f_j(\zeta)}e_j^*}{|f(\zeta)|^{2}}
\end{equation}
and 
$$
\al\widetilde\Theta_E =\sum_i\al\Theta_{L^d}\w
e_i\w e_i^*=\sum_1^m d\omega\w e_j\w e_i^*=d \omega\w\tilde I.
$$

\begin{lma}\label{l2}
We have  that 
$$
Df=\sum_1^m |z|^{2d_i}\partial\frac{f_i}{|z|^{2d_i}}\w e_i, \quad 
\dbar\sigma =\frac{1}{\|f\|^2}\sum_1^p\dbar\frac{\bar f_i}{|z|^{2d_i}}\w e_i^*+ \cdots.
$$
where $\cdots$ denote terms that contain the factor $\sigma$.
If all $d_j=d$, then 
$$
Df=\sum_1^m d f_j\w e_j+\cdots, \quad 
\dbar\sigma=\frac{1}{|f(\zeta)|^2}\sum_1^m d\bar f_j \w e^*+\cdots,
$$
where $\cdots$ denote terms that contain the factor $\sigma$ or $f$.
\end{lma}

\begin{proof}
Since 
$$
Df=\sum_i D_{L^{d_j}}f_i\w e_j,
$$
and $f$ is holomorphic, the first equality  follows from Lemma~\ref{linje}.
The second  equality is immediate in view of \eqref{citrus}. 
The two remaining equalities follow in a similar way. 
\end{proof}

Notice that $f\w\sigma$ is of the form $\alpha/\|f\|^2$ where $\alpha$
is a smooth form that is $\Ok(\|f\|^2)$.
Because of the presence of this factor in \eqref{wdef} we can insert the right hand
sides in Lemma~\ref{l2} into   \eqref{wdef}, and we then get the first
formula in Proposition~\ref{hoppla}

Now assume that $d_i=d$.
At a given  fixed point, one can assume, after an isometric
transformation, and by  homogeneity, that $f\cdot e=e_1$, so that, e.g., 
$$
dd^c\log|f|=\al \sum_2^m df_i\w d\bar f_i, \quad f\w\sigma=e_1\w e_1^*.
$$
One then obtains \eqref{polly} by a straight-forward computation.    
\end{proof}

\begin{remark}
One can actually reduce to the case when  all $d_j$ are the same. 
One  just replaces  $f_j$ by  $f_j^{r_j}$ so that $r_jd_j=d$. Then they define
the cycle $r_1\cdots r_m$ times the cycle defined by $f_j$; this follows,
e.g., from \eqref{bez}, and hence one  get a Green current for the original
cycle by just dividing by this number.
\end{remark}

Let us finally consider the case when $p<m$. Then 
$$
dd^c W_{p-1}+[Z^p]=c_p(D_E)-C_p(D_Q),
$$
and unfortunately $C_p(D_Q)$ is only the pushforward of a smooth form. However,
if we take
$$
G=W_{p-1}-\int_{\P^n} K\w C_p(D_Q)
$$
it follows from the Koppelman formula \eqref{kman}, applied to $\theta=C_p(D_Q)$,
that 
$$
dd^cG+[Z^p]=c_p(D_E)-\int_{\P^n} P\w c_p(D_Q).
$$
Thus we get a Green current whose leading term $W_{p-1}$ is quite explicit;
by a similar computation as above we find that 
\begin{equation}\label{slut}
W_{p-1}=\log(1/\|f\|)\sum_{k=1}^{p}\frac{\gamma^1_k}{\|f\|^{2k}}+
\sum_{k=1}^{p}\frac{\gamma_k^2}{\|f\|^{2k}},
\end{equation}
where $\gamma_k^i$ are smooth forms that are  $\Ok(\|f\|^2)$.

\begin{remark}
The point above was that we had quite explicit currents  $w$ and $\gamma$
such  that 
\begin{equation}\label{melog}
dd^cw+[Z^p]=\gamma,
\end{equation}
where $w$ is locally integrable, smooth outside $Z$ and
$\gamma$ is the push-forward of a smooth form.
Such currents are also provided
by a  variant of King's formula due to Meo,  \cite{Meo1} and \cite{Meo2};
in fact one can take
$$
w=-\log\|f\|\big (dd^c\log\|f\|)^{p-1}{\bf 1}_{X\setminus Z}\big)
$$
and 
$$
\gamma=(dd^c\log\|f\|)^p{\bf 1}_{X\setminus Z}.
$$
For a simple proof,  see (the proof of) Proposition~4.1 in \cite{A3}.
From that proof it follows that $w$ is of logarithmic type along $Z$ and that
$\gamma$ is the push-forward of a smooth closed form.
However, this current $w$ is not identical to \eqref{slut}; e.g., $W_{p-1}$ is positive,
whereas $w$ is not. 
\end{remark}

\def\listing#1#2#3{{\sc #1}:\ {\it #2},\ #3.}


\begin{thebibliography}{9999}


%


\bibitem{A2}\listing{M.\ Andersson}
{Residue currents and ideals of holomorphic functions}
{Bull.\ Sci.\ Math.\  {\bf 128} (2004), 481--512}


\bibitem{A3}\listing{M.\ Andersson}
{Residue currents of holomorphic sections and Lelong currents}
{Arkiv  mat.   {\bf 43} (2005), 201--219}


\bibitem{Aarxiv}\listing{M.\ Andersson}
{A generalized Poincar\'e-Lelong formula and Green currents}
{ArXiv  math.CV/0412446}{}{}{}



\bibitem{Apl}\listing{M.\ Andersson}
{A generalized Poincar\'e-Lelong formula}
{Math.\ Scand. {\bf 101} (2007), 195--218}





\bibitem{BY1}\listing{C.\ Berenstein \& A.\ Yger}
{Green currents and analytic continuation}
{J.\ Anal.\ Math.\ {\bf 75} (1998), 1--50}


\bibitem{BY2}\listing{C.\ Berenstein \& A.\ Yger}
{Analytic residue theory in the non-complete intersection case}
{J.\  Reine Angew.\  Math.\ {\bf 527} (2000),
   203--235}




\bibitem{BB}\listing{B.\ Berndtsson}
{Cauchy--Leray forms and vector bundles}
{Ann.\ Scient.\ Ec. Norm.\ Sup. {\bf 24} (1991), 319-337}


\bibitem{BGS}\listing{J-B Bost \& H  Gillet \&       C  Soul\'e}
{Heights of projective varieties and positive Green forms}
{J. Amer. Math. Soc.  {\bf 7}  (1994), 903--1027}



\bibitem{BC}\listing{R.\ Bott \& S.-S.\ Chern}{Hermitian vector bundles and the equidistribution
of the zeroes of their holomorphic sections}{Acta Math.\ {\bf 114} (1965), 71--112}



\bibitem{C}\listing{S.-S.\ Chern}
{Transgression in associated bundles}{Internat.\ J.\ Math.\  {\bf 2}  (1991),  383--393}


\bibitem{Dem}\listing{J-P Demailly}
{Complex Analytic and Differential Geometry}
{Monograph  Grenoble (1997)}



\bibitem{DiSi}\listing{T-C Dinh \& N Sibony}
{Super-potentials of positive closed currents, intersection theory and dynamics}
{ArXiv 0703702}{}


\bibitem{DiSi2}\listing{T-C Dinh \& N Sibony}
{Super-potentials for currents on compact Kaehler manifolds and dynamics of automorphisms}
{arXiv:0804.0860 }{}







\bibitem{GS}\listing{H Gillet \&  Soul\'e}
{Arithmetic intersection theory}
{Inst  Hautes \'Etudes Sci Publ Math {\bf 72} (1990),
94--174}



\bibitem{Hein}\listing{G.\ Hein}
{Green currents on K\"ahler manifolds.  The ubiquitous heat kernel}
{245--256, Contemp. Math., 398, Amer. Math. Soc., Providence, RI, 2006}
{}









\bibitem{King}\listing{J.\ R.\ King}{A residue formula for complex
subvarieties}
{Proc.\ Carolina conf.\ on holomoprhic mappings and minimal surfaces,
Univ.\ of North Carolina, Chapel Hill (1970), 43--56}







\bibitem{Meo1}\listing{M.\ Meo}{R\'esidus dans le cas non
n\'ecessairment intersection compl\`ete}
{C.\ R.\ Acad.\ Sci.\ Paris S\'er I Math.\ {\bf 333} (2001), 33-38}



\bibitem{Meo2}\listing{M.\ Meo}{Courants r\'esidus 
et formule
de King}{Arkiv  mat. {\bf 44} (2006), 149--165}









\bibitem{PTY}\listing{M.\ Passare \& A.\ Tsikh \&  A.\ Yger}
{Residue currents of the Bochner-Martinelli type}
{Publ.\ Mat.  {\bf 44} (2000), 85-117}


\bibitem{SABK}\listing{C.\ Soul\'e \& D.\ Abramovich \& J.-F.\ Burnol \&
J.\ Kramer}
{Lectures on Arakelov Geometry}
{Cambridge studies in advanced mathematics {\bf 33} 
Cambridge University Press (1992)}





\end{thebibliography}
\end{document}